\documentclass[12pt]{article}
\usepackage{graphicx}
\usepackage{amsmath,amsthm,amssymb,enumerate}
\usepackage{euscript,mathrsfs}
\usepackage{color}
\usepackage{dsfont}
\usepackage[left=2cm,right=2cm,top=3.5cm,bottom=3.5cm]{geometry}
\usepackage{color}
\usepackage{esint}
\usepackage[framemethod=tikz]{mdframed}
\allowdisplaybreaks

\usepackage{soul}

\catcode`\@=11 \@addtoreset{equation}{section}

\catcode`\@=12

\allowdisplaybreaks

\newtheorem{Theorem}{Theorem}[section]
\newtheorem{Proposition}[Theorem]{Proposition}
\newtheorem{Lemma}[Theorem]{Lemma}
\newtheorem{Corollary}[Theorem]{Corollary}

\theoremstyle{definition}
\newtheorem{Definition}[Theorem]{Definition}

\newtheorem{Remark}[Theorem]{Remark}

\newcommand{\bTheorem}[1]{
\begin{Theorem} \label{T#1} }
\newcommand{\eT}{\end{Theorem}}

\newcommand{\bProposition}[1]{
\begin{Proposition} \label{P#1}}
\newcommand{\eP}{\end{Proposition}}

\newcommand{\bLemma}[1]{
\begin{Lemma} \label{L#1} }
\newcommand{\eL}{\end{Lemma}}

\newcommand{\bCorollary}[1]{
\begin{Corollary} \label{C#1} }
\newcommand{\eC}{\end{Corollary}}

\newcommand{\bRemark}[1]{
\begin{Remark} \label{R#1} }
\newcommand{\eR}{\end{Remark}}

\newcommand{\bDefinition}[1]{
\begin{Definition} \label{D#1} }
\newcommand{\eD}{\end{Definition}}

\newcommand{\Del}{\Delta_x}

\newcommand{\bfphi}{\boldsymbol{\varphi}}

\newcommand{\bFormula}[1]{
\begin{equation} \label{#1}}
\newcommand{\eF}{\end{equation}}

\newcommand{\Ov}[1]{\overline{#1}}

\newcommand{\vt}{\vartheta}

\newcommand{\vc}[1]{{\bf #1}}

\newcommand{\Div}{{\rm div}_x}
\newcommand{\Grad}{\nabla_x}

\newcommand{\dx}{\,{\rm d} {x}}

\newcommand{\dt}{\,{\rm d} t }

\newcommand{\intO}[1]{\int_{\Omega} #1 \ \dx}
\newcommand{\avintO}[1]{\fint_{\Omega} #1 \dx}

\newcommand{\vv}{\vc{v}}

\newcommand{\D}{{\rm d}}

\newcommand{\br}{ \nonumber \\ }

\def\softd{{\leavevmode\setbox1=\hbox{d}%
          \hbox to 1.05\wd1{d\kern-0.4ex{\char039}\hss}}}
\definecolor{Cgrey}{rgb}{0.85,0.85,0.85}
\definecolor{Cblue}{rgb}{0.50,0.85,0.85}
\definecolor{Cred}{rgb}{1,0,0}
\definecolor{fancy}{rgb}{0.10,0.85,0.10}

\newcommand\Cbox[2]{%
    \newbox\contentbox%
    \newbox\bkgdbox%
    \setbox\contentbox\hbox to \hsize{%
        \vtop{
            \kern\columnsep
            \hbox to \hsize{%
                \kern\columnsep%
                \advance\hsize by -2\columnsep%
                \setlength{\textwidth}{\hsize}%
                \vbox{
                    \parskip=\baselineskip
                    \parindent=0bp
                    #2
                }%
                \kern\columnsep%
            }%
            \kern\columnsep%
        }%
    }%
    \setbox\bkgdbox\vbox{
        \color{#1}
        \hrule width  \wd\contentbox %
               height \ht\contentbox %
               depth  \dp\contentbox
        \color{black}
    }%
    \wd\bkgdbox=0bp%
    \vbox{\hbox to \hsize{\box\bkgdbox\box\contentbox}}%
    \vskip\baselineskip%
}

\mdfdefinestyle{MyFrame}{%
	linecolor=black,
	outerlinewidth=1pt,
	roundcorner=5pt,
	innertopmargin=\baselineskip,
	innerbottommargin=\baselineskip,
	innerrightmargin=10pt,
	innerleftmargin=10pt,
	backgroundcolor=white!20!white}

\date{}

\begin{document}


\title{The Oberbeck--Boussinesq system with non--local boundary conditions}

\author{Anna Abbatiello \thanks{The work of A.~A. was supported by the ERC-STG Grant n. 759229 HiCoS ``Higher Co-dimension Singularities: Minimal Surfaces and the Thin Obstacle Problem".}
 \and Eduard Feireisl
\thanks{The work of E.F. was partially supported by the
Czech Sciences Foundation (GA\v CR), Grant Agreement
21--02411S. The Institute of Mathematics of the Academy of Sciences of
the Czech Republic is supported by RVO:67985840. }
}

\date{}

\maketitle

\centerline{Sapienza University of Rome, Department of Mathematics ``G. Castelnuovo",}

\centerline{Piazzale Aldo Moro 5, 00185 Rome, Italy.}

\centerline{Email: {\tt anna.abbatiello@uniroma1.it}}

\medskip

\centerline{Institute of Mathematics of the Academy of Sciences of the Czech Republic}

\centerline{\v Zitn\' a 25, CZ-115 67 Praha 1, Czech Republic}

\centerline{Email: {\tt feireisl@math.cas.cz}}

\begin{abstract}
	
	We consider the Oberbeck--Boussinesq system with non--local boundary conditions arising as a singular limit of the full Navier--Stokes--Fourier system in the regime of low Mach and low Froude number. The existence of strong solutions is shown on a maximal time interval $[0, T_{\rm max})$. Moreover, 
	$T_{\rm max} = \infty$ in the two dimensional setting.

\end{abstract}


{\bf Keywords:} Oberbeck--Boussinesq system, non--local boundary condition, strong solution

\bigskip

\centerline{\it Dedicated to Constantine Dafermos of the occassion of his 80-th birthday}


\section{Introduction}
\label{i}

We consider the Oberbeck--Bousinesq system describing the motion of an incompressible viscous and heat conducting fluid confined to a bounded domain $\Omega \subset R^d$, $d=2,3$. The velocity $\vv = \vv(t,x)$ of the fluid and the temperature $\Theta=\Theta (t,x)$ satisfy the following system of equations:

\begin{mdframed}[style=MyFrame]
	\begin{align}
		\partial_t \vv + \Div (\vv \otimes \vv) + \Grad \Pi &= \mu \Del \vv - \Theta \Grad G, \label{i2}\\
			\Div \vv &= 0, \label{i1}\\
		\partial_t \Theta + \Div (\Theta \vv) + a \Div (G \vv) &= \kappa \Del \Theta; \label{i3}
		\end{align}
supplemented with the no-slip boundary condition for the velocity
\begin{equation} \label{i4}						
	\vv|_{\partial \Omega} = 0; 
	\end{equation}
and the non--local Dirichlet boundary condition for the temperature
\begin{equation} \label{i5}	
\Theta|_{\partial \Omega} = \Theta_B - \lambda  \avintO{ \Theta }.
\end{equation}
\end{mdframed}

\noindent The problem \eqref{i1}--\eqref{i5} has been identified as a singular limit of the full Navier--Stokes--Fourier system in the low Mach and low Froude number regime, see \cite{BelFeiOsch}. In this context, the temperature $\Theta$ is interpreted as a deviation from a background temperature and as such need not be non--negative. The symbol $G = G(x)$ stands for the gravitational potential and $\Theta_B = \Theta_B(x)$
is the prescribed boundary temperature. The transport coefficients $\mu$ and $\kappa$ are positive constants, 
$a \in R$ and $\lambda > 0$. In addition, we assume
\begin{equation} \label{i6}
\avintO{ G } = 0,\ \Del G = 0.	
	\end{equation}

There is a vast amount of literature concerning the Oberbeck--Boussinesq system, see e.g. Constantin and Doering \cite{ConDoe}, Foias, Manley and 
Temam \cite{FoMaTe}, Li and Titi \cite{LiTi} or the survey by Zeytounian \cite{ZEY1}, and the references therein. The main novelty of the present paper is  the non--local term in \eqref{i5} arising in the singular limit of the complete Navier--Stokes--Fourier system. 

Parabolic problems with non--local boundary conditions have been studied by a number of authors, see e.g.
Day \cite{Day1982}, Friedman \cite{Friedman86}, and, more recently, Pao \cite{Pao} among others. Unfortunately, most of 
the results concern the case $|\lambda| \leq 1$, while such a restriction cannot be justified in the singular limit process. Accordingly, we focus on the general case $\lambda > 0$. In particular, solutions of 
\eqref{i3}, \eqref{i5} may not comply with the standard maximum principle.

The plan of the paper is as follows. First, following \cite[Chapter 5]{FeiNovOpen}, we construct a weak solution of the problem for any finite energy initial data. These solutions are global in time for both $d=2$ and $d=3$. In addition, we show that the temperature $\Theta$ of any weak solution is necessarily bounded, 
see Section \ref{ws}. As a result, we obtain weak solutions of the Navier--Stokes system \eqref{i1}, \eqref{i2}
driven by a bounded force. Applying the $L^p$ regularity theory due to Gerhardt \cite{Gerh}, Giga and Miyakawa
\cite{GigMiy},
Solonnikov \cite{Solon}, von Wahl \cite{vonWahl}, together with the abstract weak--strong uniqueness result \cite{AbbFei2}, 
we conclude that the velocity field $\vv$ is in fact regular. This observation together with a standard bootstrap argument yields regularity of the temperature, and, consequently, the existence of strong solutions 
defined on some maximal time interval $[0, T_{\rm max})$. In addition, $T_{\rm max} = \infty$ if $d=2$,
see Section \ref{ss}.  Finally, repeating the bootstrap argument, we conclude the problem admits a classical 
smooth solutions as soon as the data are smooth satisfying the associated compatibility conditions, 
see Section \ref{cs}.
 
\section{Weak solutions}
\label{ws}

Without loss of generality, we may assume $a=0$. Indeed replacing $\Theta \approx \Theta + a G$ and using hypothesis \eqref{i6} we obtain a new system 
\begin{align}
	\partial_t \vv + \Div (\vv \otimes \vv) + \Grad \Pi &= \mu \Del \vv - \Theta \Grad G, \label{a2}\\
	\Div \vv &= 0, \label{a1}\\
	\partial_t \Theta + \Div (\Theta \vv)   &= \kappa \Del \Theta, \label{a3}
	\\
	\vv|_{\partial \Omega} &= 0, \label{a4} \\
	\Theta|_{\partial \Omega} &= \Theta_B - \lambda  \avintO{ \Theta },
	\label{a5}
\end{align}
with $\Theta_B \approx \Theta_B + a G.$

\begin{Definition}[{\bf Weak solution}] \label{sD1}
	We say that $\vv$, $\Theta$ is \emph{weak solution} to problem \eqref{a1}--\eqref{a5} 
	with the initial data 
	\[
	\vv(0, \cdot) = \vv_0,\ \Theta(0, \cdot) = \Theta_0
	\]
	if the following holds:
	
	\begin{itemize}
		\item Regularity.
		\begin{align} 
			\vv &\in L^\infty(0,T; L^2(\Omega; R^d)) \cap 
			L^2(0,T; W^{1,2}_0 (\Omega; R^d)), \br
		\Theta &\in L^\infty((0,T) \times \Omega) \cap L^2(0,T; W^{1,2}(\Omega)).
		\label{s4}
	\end{align}
	\item Equations of motion. 
\begin{align} 
\int_0^T &\intO{ \Big[ \vv \cdot \partial_t \bfphi + (\vv \otimes \vv): \Grad \bfphi \Big] } \dt \br 
&= \int_0^T \intO{ \mu \Grad \vv : \Grad \bfphi } \dt + \int_0^T \intO{ \Theta \Grad G \cdot \bfphi } \dt
- \intO{ \vv_0 \cdot \bfphi(0, \cdot) }
\label{ws2}	
	\end{align}
for any $\bfphi \in C^1_c([0,T) \times \Omega; R^d)$, $\Div \bfphi = 0$,
\begin{equation} \label{ws1}
	\Div \vv = 0 \ \mbox{a.a. in}\ (0,T) \times \Omega.
	\end{equation}
\item Heat equation.
		\begin{equation} \label{ws3}
		\int_0^T \intO{ \Big[ \Theta \partial_t \varphi + \Theta \vv \cdot \Grad \varphi \Big] } \dt 
		= \int_0^T \intO{ \kappa \Grad \Theta \cdot \Grad \varphi } \dt - \intO{ \Theta_0 \varphi (0, \cdot) }
		\end{equation}
		for any $\varphi \in C^1_c([0,T) \times \Omega)$.
\item Boundary conditions.		 	
		\begin{equation} \label{ws4}
		\Theta  + \lambda \avintO{ \Theta } - \Theta_B \in L^2(0,T; W^{1,2}_0 (\Omega)).
		\end{equation}
\item Mechanical energy inequality.
\begin{equation} \label{ws5}
	\frac{1}{2} \intO{ |\vv|^2 (\tau, \cdot) } + \mu \int_0^\tau \intO{ |\Grad \vv |^2 } \dt \leq 
		\frac{1}{2} \intO{ |\vv_0|^2  } - \int_0^\tau \intO{ \Theta \Grad \vc{G} \cdot \vv } \dt 
	\end{equation}
for any $0 \leq \tau \leq T$.
	\end{itemize} 
\end{Definition}

\begin{Remark} \label{wsR1}
	
	In \eqref{ws4}, we tacitly assume that $\Theta_B$ has been suitably extended inside $\Omega$.
	
	\end{Remark}

\subsection{Existence of weak solution} 

Our first result asserts global--in--time existence of weak solutions.

\begin{mdframed}[style=MyFrame]
	\begin{Theorem}[\bf Weak solutions] \label{MT1}
		
		Suppose that $\Omega \subset R^d$, $d=2,3$ is a bounded Lipschitz domain, and 
		\[
		G \in W^{1, \infty}(\Omega),\ \Theta_B \in C(\Ov{\Omega}) \cap W^{1,2}(\Omega).
		\]
		 Let the initial data satisfy 
		\begin{align} 
			\vv_0 &\in L^2(\Omega; R^d),\ \Div \vv_0 = 0,\ \vv_0 \cdot \vc{n}|_{\partial \Omega} = 0,\br
			\Theta_0 &\in C(\Ov{\Omega}),\ \Theta_0 + \lambda \avintO{ \Theta_0 } = \Theta_B \ \mbox{on}\ 
			\partial \Omega.
			\label{M1}
		\end{align}
	
		Then for any $T > 0$, the problem \eqref{a1}--\eqref{a5} admits a weak solution $\vv$, $\Theta$ in the sense of Definition \ref{sD1}.

	\end{Theorem}

\end{mdframed}

The rest of this section is devoted to the proof of Theorem \ref{MT1}.

\subsection{Faedo--Galerkin approximation}

First, suppose that $\partial \Omega$ is smooth. 
Similarly to \cite[Chapter 5]{FeiNovOpen}, approximate solutions to problem \eqref{a1}--\eqref{a5} are constructed by means of the Faedo--Galerkin approximation of the momentum equation \eqref{a2} and exact 
solutions to the heat equation \eqref{a3}, \eqref{a5}. 

Let $\{ \vc{w}_n \}_{n=1}^\infty$ be the orthogonal basis of the Hilbert space of solenoidal functions $L^2_{\sigma}(\Omega; R^d)$ formed by the 
eigenfunctions of the Stokes operator, with the associated orthogonal projections $\Pi_N$,
\[
\Pi_N : L^2_\sigma \to X_N \equiv {\rm span} \{ \vc{w}_1, \dots, \vc{w}_N \},\ N=1,2,\dots.
\]

The approximate velocity fields $\vv_N \in C^1([0,T]; X_N)$ are obtained as solutions of the system of ordinary differential equations:
\begin{equation} \label{w1}
	\partial_t \vv_N + \Pi_N [\Div(\vv_N \otimes \vv_N)] = \mu \Del \vv_N - \Pi_N [\Theta \Grad G],\ 
	\vv_N(0, \cdot) = \Pi_N \vv_0,
	N=1,2,\dots
\end{equation}

The approximate temperature $\Theta = \Theta_N$ is \emph{exact} solution to the problem 
\begin{equation} \label{ws6}
	\partial_t \Theta_N + \Div (\vv_N \Theta_N ) = \kappa \Del \Theta_N,\ 
	\Theta_N|_{\partial \Omega} = \Theta_B - \lambda \avintO{ \Theta_N },\ 
	\Theta_N (0, \cdot) = \Theta_0.	
\end{equation}

\begin{Lemma}[{\bf Uniqueness for the heat equation}]\label{sL1}
	Given a velocity 
	\[
	\vv \in L^\infty((0,T) \times \Omega; R^d),
	\]
	there is at most one weak solution $\Theta$ of the problem
	\begin{equation} \label{ws7}
		\partial_t \Theta + \Div (\vv \Theta ) = \kappa \Del \Theta,\ 
		\Theta|_{\partial \Omega} = \Theta_B - \lambda \avintO{ \Theta },\ 
		\Theta (0, \cdot) = \Theta_0,	
	\end{equation}
in the sense specified in \eqref{ws3}, \eqref{ws4} in Definition \ref{sD1} in the class
\[
\Theta \in L^\infty(0,T; L^2(\Omega)) \cap L^2(0,T; W^{1,2}(\Omega))
\]
\end{Lemma}

\begin{proof}
Let $\Theta_1$, $\Theta_2$ be two solutions with the same initial value $\Theta_0$. Accordingly, 
\[
\hat \Theta = \Theta_1 - \Theta_2 
\]
solves the problem 
\[
	\partial_t \hat \Theta + \Div ( \hat \Theta \vv) = \kappa \Del \hat \Theta,\ 
	\hat \Theta|_{\Omega} =  - \lambda \avintO{ \hat \Theta }, 
	\
	\hat \Theta(0, \cdot) = 0.
	\]
Thus, introducing a new quantity 
\[
V = \hat \Theta + \lambda \avintO{ \hat \Theta},
\]
we obtain 
\begin{equation} \label{ws8}
\partial_t \left( V - \frac{\lambda}{1 + \lambda} \avintO{ V } \right) + \Div (\vv V) = 
\kappa \Del V,\ V|_{\partial \Omega} = 0,\ V(0, \cdot) = 0,
\end{equation}
where the differential equation is satisfied in the sense of distributions. As 
\[
V \in L^\infty(0,T; L^2(\Omega)) \cap L^2(0,T; W^{1,2}(\Omega)), 
\]	
we may justify multiplication of equation \eqref{ws8} on $V$ and subsequent integration by parts yielding
\[
\frac{1}{2} \left[ \avintO{ |V|^2 (\tau, \cdot)}  - \frac{\lambda}{1 + \lambda} \left( \avintO{ 
V } \right)^2 \right] 
+ \kappa \int_0^\tau \avintO{ |\Grad V|^2 } = 0 \ \mbox{for any}\ 0 \leq \tau \leq T.
\]	
Consequently $V = 0 \ \Rightarrow \ \hat \Theta = 0$.	
	\end{proof}

In view of Lemma \ref{sL1}, we may consider a well defined mapping 
\[
\vv_N \mapsto \Theta (\vv_N) = \Theta_N,
\]
where $\Theta_N$ is the unique solution of \eqref{ws6} as soon as we show such a solution exists. 

The proof of solvability of \eqref{ws6} for a given $\vv_N$ follows the same line of arguments as the proof of 
Lemma \ref{sL1}. First, we introduce 
\[
V = \Theta + \lambda \avintO{ \Theta }
\]
transforming the problem to 
\begin{equation} \label{ws9}
\partial_t \left( V - \frac{\lambda}{1 + \lambda} \avintO{ V } \right) + \Div (\vv_N V) = 
\kappa \Del V,\ V|_{\partial \Omega} = \Theta_B,\ V(0, \cdot) = \Theta_0 + \lambda \avintO{ \Theta_0 }.	
	\end{equation}
Solutions of \eqref{ws9} may be written in the form 
\[
V = Z + \Theta_B, 
\]
where 
\begin{equation} \label{ws10}
	\partial_t \left( Z - \frac{\lambda}{1 + \lambda} \avintO{ Z } \right) + \Div (\vv_N (Z + \Theta_B)) = 
	\kappa \Del Z,\ Z|_{\partial \Omega} = 0,\ Z(0, \cdot) = \Theta_0 + \lambda \avintO{ \Theta_0 } - \Theta_B.	
\end{equation}
Here, for the sake of simplicity, we have considered the harmonic extension of $\Theta_B$ inside $\Omega$.
Seeing that the mapping 
\[
\mathcal{L}[V] = V - \frac{\lambda}{1 + \lambda} V,\ \lambda > 0 
\ \mbox{is a self--adjoint isomphism on the Hilbert space}\ L^2(\Omega),
\]
we may solve \eqref{ws10} by a Faedo--Galerkin method based on the system of eigenfuctions of the 
Dirichlet Laplacean on the domain $\Omega$. Note that this approximation is compatible with the associated 
``energy'' estimates based on multiplication of \eqref{ws10} on $Z$:
\begin{align} 
\frac{1}{2} &\left[ \avintO{ |Z|^2 (\tau, \cdot)}  - \frac{\lambda}{1 + \lambda} \left( \avintO{ 
	Z } \right)^2 \right]
+ \kappa \int_0^\tau \avintO{ |\Grad Z|^2 } \br &\leq 
\frac{1}{2}\left[ \avintO{ |Z_0|^2 }  - \frac{\lambda}{1 + \lambda} \left( \avintO{ 
	Z_0} \right)^2 \right]  + \int_0^\tau \intO{ \Theta _B \vv_N \cdot \Grad Z } 
\label{ws11}
\end{align}
for any $0 \leq \tau \leq T$, where 
\[
Z_0 = \Theta_0 + \lambda \avintO{ \Theta_0 } - \Theta_B.
\]

Inequality \eqref{ws11} yields the desired uniform bounds so that the Faedo--Galerkin approximation converges to a unique solution in the class 
\[
Z \in L^\infty(0,T; L^2(\Omega)) \cap 
L^2(0,T; W^{1,2}_0 (\Omega)) \ \Rightarrow \ 
\Theta_N \in L^\infty(0,T; L^2(\Omega)) \cap 
L^2(0,T; W^{1,2} (\Omega)).
\]

It follows from the previous discussion that local-in-time existence of the approximate solutions 
$\vv_N$, $\Theta_N$ can be established exactly as in \cite[Chapter 5]{FeiNovOpen}. Moreover, 
multiplying the approximate momentum equation on $\vv_N$ we obtain an approximate version of the energy 
balance, 
\begin{equation} \label{ws12}
	\frac{1}{2} \intO{ |\vv_N|^2 (\tau, \cdot) } + \mu \int_0^\tau \intO{ |\Grad \vv_N |^2 } \dt \leq 
\frac{1}{2} \intO{ |\vv_0|^2  } - \int_0^\tau \intO{ \Theta_N \Grad \vc{G} \cdot \vv_N } \dt,
\end{equation}
which, together with \eqref{ws11} written in terms of $\Theta_N$, 
\begin{align} 
	\frac{1}{2} &\frac{1}{1 + \lambda}\intO{ \left| \Theta_N + \lambda \avintO{ \Theta_N} - \Theta_B \right|^2 (\tau, \cdot) } 
	+ \kappa \int_0^\tau \intO{ |\Grad( \Theta - \Theta_B) |^2 } \br &\leq 
	\frac{1}{2}  \intO{ \left| \Theta_0 + \lambda \avintO{ \Theta_0 } - \Theta_B \right|^2 } + \int_0^\tau \intO{ \Theta _B \vv_N \cdot \Grad(\Theta - \Theta_B) },
	\label{ws13}
\end{align}
yield the desired uniform bounds that guarantee global existence of the approximate solutions up to any desired time $T > 0$.

Finally, as the approximate velocities $\vv_N$ are smooth, the standard parabolic local estimates yield smoothness of the approximate temperatures $\Theta_N$ in the interior of the space--time cylinder 
$(0,T) \times \Omega$. Moreover, thanks to the compatibility condition \eqref{M1} imposed on the 
initial data, the approximate temperatures are continuous up to the boundary of $[0,T] \times \Ov{\Omega}$. 
In particular, the standard maximum principle yields 
\begin{equation} \label{a14}
	\underline{\Theta} \leq \Theta_N (t,x) \leq \Ov{\Theta} \ \mbox{for}\ (t,x) \in [0,T] \times \Omega, 
\end{equation}
for some constants $\underline{\Theta}$, $\Ov{\Theta}$ depending only on the initial and boundary data and $T$.

With the uniform bounds \eqref{ws12}--\eqref{a14}, it is a routine matter to perform the limit $N \to \infty$ 
in the family of approximate solutions $( \vv_N; \Theta_N )_{N=1}^\infty$ and to show that the limit yields 
the desired weak solution the existence of which is claimed in Theorem \ref{MT1}. Note that, by means of the 
standard Aubin--Lions argument, 
\[
\Theta_N \to \Theta \ \mbox{in}\ L^2((0,T) \times \Omega) 
\ \mbox{and weakly in}\ L^2(0,T; W^{1,2}(\Omega));
\]
which, by interpolation, yields the convergence of traces on $\partial \Omega$. We have proved 
Theorem \ref{MT1}. 

\begin{Remark} \label{wsR2}
	
	The reader will have noticed that Theorem \ref{MT1} requires the domain $\Omega$ to be merely Lipschitz, while the proof via Faedo--Galerkin approximation was based on smooth domains. However, the result can be extended to Lipschitz domains by their approximation by smooth ones. The details are left to the reader.
	
	\end{Remark}

\section{Strong solutions}
\label{ss}

Our next goal is to improve regularity of the weak solutions on condition that the data are smooth.
 
\begin{mdframed}[style=MyFrame]
	\begin{Theorem}[\bf Strong solutions] \label{MT2}
		
		In addition to the hypotheses of Theorem \ref{MT1}, suppose
		that $\Omega \subset R^d$, $d=2,3$ is a bounded domain of class $C^2$, 
		\begin{equation} \label{ss1}
		G \in W^{1, \infty}({\Omega}),\ \Theta_B \in C^2(\Ov{\Omega}),	
			\end{equation}
		and the initial data satisfy 
		\begin{align} 
			\Theta_0 \in W^{2,p}(\Omega),\ \vv_0 \in W^{2,p}(\Omega; R^d),\ \Div \vv_0 = 0, 
			\ \mbox{for any}\ 1 \leq p < \infty, \br
			\mbox{together with the compatibility conditions} \br
			\vv_0 = 0,\ \Theta_0 + \lambda \avintO{ \Theta_0 }|_{\partial \Omega} = 
			\Theta_B \ \mbox{on}\ \partial \Omega
			\label{M3}
		\end{align}
		
		Then there exists $T_{\rm max} > 0$, $T_{\rm max} = \infty$ if $d = 2$, such that
		any weak solution $\vv$, $\Theta$ belongs to the regularity class
		\begin{align} 
			\vv \in L^p(0,T; W^{2,p}(\Omega; R^d)),\ \partial_t \vv \in L^p(0,T; L^{p}(\Omega; R^d)), \br 
			\Theta \in L^p(0,T; W^{2,p}(\Omega)),\ 
			\partial_t \Theta \in L^p(0,T; L^{p}(\Omega; R^d)) \ \mbox{for any}\ 1 \leq p < \infty
			\label{M4}	
		\end{align}	
		for any $0 < T < T_{\rm max}$.	
		
	\end{Theorem}
	
\end{mdframed}

The rest of this section is devoted to the proof of Theorem \ref{MT2}. 

\subsection{Regularity of the velocity field}

As $\Theta \in L^\infty((0,T) \times \Omega)$ the velocity field $\vv$ solves the Navier--Stokes system \eqref{a1}, \eqref{a2} with a bounded driving force $\Theta \Grad G$, in particular, 
$\Theta \Grad G \in L^p((0,T) \times \Omega; R^d)$ for any $1 \leq p < \infty$. In view of the 
available $L^p-$ theory for the Stokes system due to Solonnikov \cite{Solon} and its adaptation to the Navier--Stokes system by Giga and Miyakawa \cite{GigMiy} (cf. also Bothe and Pruess \cite{BotPru}), there exists a time $T_{\rm max} > 0$ such that the Navier--Stokes system with the initial data in the class 
\eqref{M3} and the right--hand side $\Theta \Grad G$ admits a strong solution $\vv$, 
\[
\vv \in L^p(0,T; W^{2,p}(\Omega; R^d)),\ \partial_t \vv \in L^p(0,T; L^{p}(\Omega; R^d))
\]	
for any $1 \leq p < \infty$ and any $0 < T < T_{\rm max}$. Moreover, as shown by Gerhardt \cite{Gerh} and 
von Wahl \cite{vonWahl}, $T_{\rm max} = \infty$ if $d = 2$.

Finally, as any weak solution in the sense of Definition \ref{sD1} satisfies the energy inequality 
\eqref{ws5}, we conclude, using the general weak--strong uniqueness principle established in \cite{AbbFei2}, 
that the velocity component $\vv$ necessarily enjoys the regularity claimed in \eqref{M4}.

\subsection{Regularity of the temperature}

In view of the previous discussion, we already know that $\Theta$ satisfies equation \ref{a3} with \emph{regular} drift term, in particular $\Theta$ is regular inside $\Omega$. Moreover, given $\vv$, the 
solution $\Theta$ is unique in the class of weak solutions, see Lemma \ref{sL1}. Consequently, it is enough to establish {\it a priori} bounds that would guarantee Lipschitz continuity of the 
average $\avintO{ \Theta }$ in time. Indeed the abstract theory of Denk, Hieber, and Pruess \cite{DEHIEPR} 
would then guarantee the required $L^p-$regularity of $\Theta$. Note that the existing bounds established so far 
only imply 
\[
\Theta \in C_{\rm weak} ([0,T]; L^2(\Omega)) \ \Rightarrow \ 
\avintO{ \Theta } \in C[0,T].
\]

The desired {\it a priori} bound follow by differentiating equation \eqref{a3} in time. Setting 
\[
\partial_t \Theta = \vt 
\]
we obtain 
\begin{equation} \label{ss5}
	\partial_t \vt + \Div (\vv \vt) = \kappa \Del \vt - \Div (\partial_t \vv \Theta) \ \mbox{in}\ 
	(0,T) \times \Omega
	\end{equation}
with the boundary condition 
\begin{equation} \label{ss6} 
	\vt = - \lambda \avintO{ \vt } \ \mbox{on}\ \partial \Omega, 
	\end{equation}
and the initial condition 
\begin{equation} \label{ss7}
	\vt(0, \cdot) = \vt_0 = \kappa \Del \Theta_0 - \Div (\vv_0 \Theta_0 ).
\end{equation}

Similarly to the above, we consider 
\[
V = \vt + \lambda \vt, 
\]
rewriting \eqref{ss5} in the form 
\begin{equation} \label{ss8}
	\partial_t \left( V - \frac{\lambda}{1 + \lambda} \avintO{ V } \right) + \Div (\vv V) = 
	\kappa \Del V - \Div (\partial_t \vv \Theta),\ V|_{\partial \Omega} = 0.
\end{equation}
Evoking the energy estimate \eqref{ws11} we get 
\begin{equation} \label{ss9}
	\frac{1}{2} \frac{1}{1 + \lambda}\intO{ \left| V \right|^2 (\tau, \cdot) } 
	+ \kappa \int_0^\tau \intO{ |\Grad V|^2 } \leq 
	\frac{1}{2}  \intO{ |V_0 |^2 } + \int_0^\tau \intO{ \Theta \partial_t \vv \cdot \Grad V } 
\end{equation}
for any $0 \leq \tau \leq T$. Seeing that $\Theta \in L^\infty((0,T) \times \Omega)$, 
$\partial_t \vv \in L^2(0,T; L^2(\Omega; R^d))$ we deduce the desired bound 
\[
V \in L^\infty(0,T; L^2(\Omega)) \ \Rightarrow \ \partial_t \Theta 
\in L^\infty(0,T; L^2(\Omega)) \ \Rightarrow\ \avintO{ \Theta } \in W^{1,\infty}(0,T).
\]

We have proved Theorem \ref{MT2}.

\section{Classical solutions}
\label{cs}

Finally, we claim the following result that can be obtained as a consequence of the standard parabolic theory 
given the regularity of the weak solution established in Theorem \ref{MT2}.

\begin{mdframed}[style=MyFrame]
	\begin{Theorem}[\bf Classical solutions] \label{MT3}
		
		In addition to the hypotheses of Theorem \ref{MT1}, suppose
		that $\Omega \subset R^d$, $d=2,3$ is a bounded domain of class $C^{2 + \nu}$, $\nu > 0$,
		\[
		G \in C^{1 + \mu}(\Ov{\Omega}), \ \Theta_B \in C^{2 + \nu}(\Ov{\Omega}).
		\]
		 and the initial data satisfy 
		\begin{align} 
			\Theta_0 \in C^{2 + \nu}(\Omega),\ \vv_0 \in C^{2 + \nu}(\Omega; R^d),\ \Div \vv_0 = 0,  \br
			\vv_0 = 0,\ \Theta_0 + \lambda \avintO{ \Theta_0 } = 
			\Theta_B \ \mbox{on}\ \partial \Omega, \br 
			- \mu \Del \vv_0 + \Theta_0 \Grad G = - \Grad \Pi_0 \ \mbox{on} \ \partial \Omega,\br
			\kappa \Del \Theta_0 = - \frac{\lambda \kappa}{|\Omega|} \int_{\partial \Omega} \Grad \Theta_0 \cdot \vc{n} 
			\ \D \sigma_x \ \mbox{on}\ \partial \Omega. 
			\label{M5}
		\end{align}
		
		Then there exists $T_{\rm max} > 0$, $T_{\rm max} = \infty$ if $d = 2$, such that
		any weak solution $\vv$, $\Theta$
		is a classical solution, specifically, 
		\begin{align} 
			\vv , \ \nabla_x^2 \vv,\ \partial_t \vv \in C^\beta([0,T] \times \Ov{\Omega}; R^d), \br 
			\Theta, \ \nabla^2_x \Theta, \ \partial_t \Theta \in  C^\beta([0,T] \times \Ov{\Omega})
			\ \mbox{for some}\ \beta > 0,
			\label{M6}	
		\end{align}	
		for any $0 < T < T_{\rm max}$.	
		
	\end{Theorem}
	
\end{mdframed}

\bibliographystyle{plain}

\begin{thebibliography}{10}

\bibitem{AbbFei2}
A.~Abbatiello and E.~Feireisl.
\newblock On a class of generalized solutions to equations describing
  incompressible viscous fluids.
\newblock {\em Ann. Mat. Pura Appl. (4)}, {\bf 199}(3):1183--1195, 2020.

\bibitem{BelFeiOsch}
 P.~Bella, E.~Feireisl and F.~Oschmann.
\newblock{The incompresible limit for the Rayleigh--B\'{e}nard convection
problem.}
\newblock{\em {\bf arXiv preprint No. 2206.14041}}, 2022.

\bibitem{BotPru}
D.~Bothe and J.~Pr\"{u}ss.
\newblock {$L_P$}-theory for a class of non-{N}ewtonian fluids.
\newblock {\em SIAM J. Math. Anal.}, {\bf 39}(2):379--421, 2007.

\bibitem{ConDoe}
P.~Constantin and Ch.~R. Doering.
\newblock Heat transfer in convective turbulence.
\newblock {\em Nonlinearity}, {\bf 9}(4):1049--1060, 1996.

\bibitem{Day1982}
W.~A. Day.
\newblock Extensions of a property of the heat equation to linear
  thermoelasticity and other theories.
\newblock {\em Quart. Appl. Math.}, {\bf 40}(3):319--330, 1982/83.

\bibitem{DEHIEPR}
R.~Denk, M.~Hieber, and J.~Pr{\" u}ss.
\newblock Optimal ${L}^p-{L}^q$-estimates for parabolic boundary value problems
  with inhomogenous data.
\newblock {\em Math. Z.}, {\bf 257}:193--224, 2007.

\bibitem{FeiNovOpen}
E.~Feireisl and A.~Novotn{\' y}.
\newblock {\em Mathematics of open fluid systems}.
\newblock Birkh{\" a}user--Verlag, Basel, 2022.

\bibitem{FoMaTe}
C.~Foias, O.~Manley, and R.~Temam.
\newblock Attractors for the {B}\'{e}nard problem: existence and physical
  bounds on their fractal dimension.
\newblock {\em Nonlinear Anal.}, {\bf 11}(8):939--967, 1987.

\bibitem{Friedman86}
A.~Friedman.
\newblock Monotonic decay of solutions of parabolic equations with nonlocal
  boundary conditions.
\newblock {\em Quart. Appl. Math.}, {\bf 44}(3):401--407, 1986.

\bibitem{Gerh}
C.~Gerhardt.
\newblock {$L^{p}$}-estimates for solutions to the instationary
  {N}avier-{S}tokes equations in dimension two.
\newblock {\em Pacific J. Math.}, {\bf 79}(2):375--398, 1978.

\bibitem{GigMiy}
Y.~Giga and T.~Miyakawa.
\newblock Solutions in {$L_r$} of the {N}avier-{S}tokes initial value problem.
\newblock {\em Arch. Rational Mech. Anal.}, {\bf 89}(3):267--281, 1985.

\bibitem{LiTi}
J.~Li and E.~S. Titi.
\newblock Global well-posedness of the 2{D} {B}oussinesq equations with
  vertical dissipation.
\newblock {\em Arch. Ration. Mech. Anal.}, {\bf 220}(3):983--1001, 2016.

\bibitem{Pao}
C.~V. Pao.
\newblock Asymptotic behavior of solutions of reaction-diffusion equations with
  nonlocal boundary conditions.
\newblock volume~{\bf 88}, pages 225--238. 1998.
\newblock Positive solutions of nonlinear problems.

\bibitem{Solon}
V.~A. Solonnikov.
\newblock Estimates for solutions of a non-stationary linearized system of
  {N}avier-{S}tokes equations.
\newblock {\em Trudy Mat. Inst. Steklov.}, {\bf 70}:213--317, 1964.

\bibitem{vonWahl}
W.~von Wahl.
\newblock Instationary {N}avier-{S}tokes equations and parabolic systems.
\newblock {\em Pacific J. Math.}, {\bf 72}(2):557--569, 1977.

\bibitem{ZEY1}
R.~Kh. Zeytounian.
\newblock {J}oseph {B}oussinesq and his approximation: a contemporary view.
\newblock {\em C.R. Mecanique}, {\bf 331}:575--586, 2003.

\end{thebibliography}

\def\cprime{$'$} \def\ocirc#1{\ifmmode\setbox0=\hbox{$#1$}\dimen0=\ht0
  \advance\dimen0 by1pt\rlap{\hbox to\wd0{\hss\raise\dimen0
  \hbox{\hskip.2em$\scriptscriptstyle\circ$}\hss}}#1\else {\accent"17 #1}\fi}

\end{document}